\documentclass[11pt,reqno]{amsart}
\usepackage{amssymb,amsmath}\newcommand{\diam}{\mbox{\rm diam}}\newcommand{\be}{\begin{eqnarray}}
\newcommand{\ee}{\end{eqnarray}}\newcommand{\card}{\#}\newcommand{\eps}{{\mbox{$\epsilon$}}}\newcommand{\e}{{\varepsilon}}\newcommand{\R}{{\mathbb R}}\newcommand{\K}{{\mathcal K}}
\newcommand{\ch}{{\cosh}}
\newcommand{\sh}{{\sinh}}

\newtheorem{theorem}{Theorem}\newtheorem{lemma}[theorem]{Lemma}\theoremstyle{definition}\theoremstyle{remark}\numberwithin{equation}{section}\input epsf.sty
\begin{document}\thispagestyle{empty}

%
%
%
%
%
%
%
%
%
%
%
%
\newcommand{\G}{{\mathcal G}}
\newcommand{\C}{{\mathbb C}}
\newcommand{\Pa}{{P_{1,\theta}(y)}}
\newcommand{\Pb}{{P_{2,\theta}(y)}}
\newcommand{\Pshp}{{P_{2,\theta}^\sharp(y)}}
\newcommand{\Pflt}{{P_{2,\theta}^\flat(y)}}
\newcommand{\OTL}{{\Omega(\theta,\ell)}}
\newcommand{\rsz}{{R(\theta^*)}}

\title[Buffon needles landing near  Sierpinski gasket]{{Buffon needles landing near  Sierpinski gasket}}
\author{Matthew Bond}\address{Matthew Bond, Dept. of Math., Michigan State University.
{\tt bondmatt@msu.edu}}
\author{Alexander Volberg}\address{Alexander Volberg, Dept. of  Math., Michigan State Univ. 
{\tt volberg@math.msu.edu}}\,.

\thanks{Research of the authors was supported in part by NSF grants  DMS-0501067, 0758552 }
\begin{abstract}In this paper we modify the method of Nazarov, Peres, and Volberg \cite{NPV} to get an estimate from above of the Buffon needle probability of the $n$th partially constructed Sierpinski gasket of Hausdorff dimension 1.
\end{abstract}
\maketitle

\section{{\bf Introduction}} \label{sec:intro}
Among   self-similar planar sets of Hausdorff dimension $1$, the simplest are the Sierpinski gasket $\G$ (formed by three self-similarities) and the square $1/4$ corner Cantor set $\K$ (formed by four self-similarities). By the Besicovitch projection theorem, these irregular sets of positive and finite Hausdorff $H^1$ measure must have zero length in almost every orthogonal projection onto a line. One may partially construct these sets in the usual way by taking the convex hull and then taking the union of all possible images of $n$-fold compositions of the similarity maps. Then one may ask the rate at which the Favard length -- the average over all directions of the length of the orthogonal projection onto a line in that direction -- of these sets $\G_n$ and $\K_n$ decay to zero as a function of $n$.

In the case of $\K_n$, an upper bound and a lower bound are known. The lower bound was obtained relatively easily in a paper of Bateman and Volberg \cite{BV} (see also \cite{BV2} for a related question): it is $c\frac{\log\,n}{n}$. The argument painlessly yields the same lower bound for $\G_n$.

The upper bound for $\K_n$ is due to Nazarov, Peres, and Volberg \cite{NPV}: if $p<1/6$, $Fav(\K_n)\leq\frac{c_p}{n^{p}}$. To get this estimate, the radial symmetry was used in addition to a reflection symmetry which $\G_n$ lacks. The main idea of \cite{NPV}  was to split the directions into good and singular ones, and to show that the measure  of singular directions is small. This idea holds for $\G_n$, but the changes which must be made are not completely superficial. The goal of this paper is to make whatever changes are necessary to find some upper bound of the decay rate for the $n$-th partial Sierpinski gasket. The struggle, as often in analysis, is with the set of small values of a certain function (in our case here the function is an exponential polynomial). In \cite{NPV} this exponential polynomial happened to be just a sine function. The case of the gasket  is  much closer to the generic case as the polynomial becomes a rather general $3$-term exponential sum. Notice that, in fact, it is an entire function of $2$ variables: one variable is given by the choice of  the direction of projection (and in our considerations below should be made even complex by some reason!), another variable is its ``spectral" variable. Sorting out the zeros and the set of small values of this entire function will give us some headache. However, the advantage is that the Sierpinski gasket provides a much better glimpse at the general self-similar sets completely irregular in the sense of Besicovitch than the $1/4$ corner Cantor set. We believe that using this approach one can work with all such sets. We pay for that: while \cite{NPV} combined combinatorics with Fourier analyis, here we need to add a certain amount of complex analysis into reasoning. Rather strangely, a special case of the Carleson Embedding Theorem, in the form of Lemma \ref{CET}, plays an important part in our reasoning.

Notice that product structure of the underlying Cantor set was recently explored in Laba-Zhai's paper \cite{LZ}, where they extended the result of \cite{NPV} to product Cantor sets. Their argument involves a combinatorial reasoning related to tiling studied by Kenyon \cite{kenyon} and Lagarias-Wang \cite{lawang}.

Consider the function $f_{n,\theta}:\R\to\mathbb N$ defined by $$f_{n,\theta}=\sum_\text{Sierpinski triangles T}{\chi_{proj_\theta(T)}}$$ Note that $Fav(\G_n)=\pi^{-1}\int_0^{\pi}|supp(f_{n,\theta})|d\theta$. In \cite{NPV} and \cite{BV}, the $L^p$ norms of the analog of this function for squares were studied to obtain Buffon needle probability estimates for $\K_n$ -- in \cite{BV}, $p=1,2$ were related to $\chi_{supp(f_{n,\theta})}$ via the Cauchy inequality, while in \cite{NPV}, $p=2$ was studied via Fourier transforms and related to the case of $p=\infty$. Indeed, if we ignore the averaging over $\theta$ for the moment and consider a sum of characteristic functions of intervals whose $L^1$ norm is 1, then heuristically, the argument is that as the mass becomes more concentrated on smaller sets, the $L^p$ norms will grow. Thus for $p>1$, a large $L^p$ norm should indicate that the support of a function is small, and vice versa. Therefore, to show that the $Fav(\G_n)$ is small, we will show that if we fix N large, then for most angles $\theta$ it will follow that $||f_{n,\theta}||_\infty$ is large for at least one $n<N$.

In the first part of the paper, Section \ref{sec:fourier}, we will study and prove one such statement using Fourier analysis. The task of making a rigorous link between the $L^\infty$ norm of $f_{n,\theta}$ and the needle probability of $\G_n$ will be undertaken in the combinatorial part of the paper, Section \ref{combinatorics}. Many of the claims of Section \ref{sec:fourier} will rest on the complex-analytic reasoning of Section \ref{complex}. Finally, some standard lemmas will be appealed to repeatedly, which we will state and prove in Section \ref{lemmas}.

The main result of this article is the following estimate (of course far from being optimal, see Section \ref{discu} for the further discussion).
\begin{theorem}
\label{mainth1}
$$
Fav(\G_n)\le C\, e^{-c\,\sqrt{\log \log n}}\,.
$$
\end{theorem}

\medskip

\noindent{\bf Acknowledgements.} We are grateful to Fedja Nazarov for many valuable discussions concerning a general case of a self-similar set.

\section{{\bf The Fourier-analytic part}} \label{sec:fourier}
Our computations will be simplified if we first rescale $\G_n$ by a factor absolutely comparable to 1 and bound the triangles by discs and study this set instead. That is, for $\alpha\in\lbrace -1, 0, 1\rbrace^{n+1}$ let $$z_\alpha:=\sum_{k=0}^n{(\frac{1}{3})^ke^{i\pi[\frac{1}{2}+\frac{2}{3}\alpha_k]}},$$ and then let $$\G_n:=\bigcup_{\alpha\in\lbrace -1, 0, 1\rbrace^{n+1}}B(z_\alpha,3^{-n}).$$ Note that $\G_n$ has $3^{n+1}$ discs of radius $3^{-n}.$ After a rescaling, the usual $n+1$st Sierpinski gasket (composed of $3^{n+1}$ triangles) sits inside of $\G_n$. We may still speak of the approximating discs as ``Sierpinski triangles."

Observe that $f_{n,\theta}=\nu_n *3^n\chi_{[-3^{-n},3^{-n}]},$ where $\nu_n :=*_{k=0}^n\widetilde{\nu}_k$ and $$\widetilde{\nu}_k=\frac{1}{3}[\delta_{3^{-k}cos(\pi/2 -\theta)} +\delta_{3^{-k}cos(-\pi/6 -\theta)} +\delta_{3^{-k}cos(7\pi/6 -\theta)}].$$

Let us fix $p$ and let $E:=E_{K}:=\lbrace\theta : \sup_{n\le N} ||f_{n,\theta}||_{L^\infty(x)} \leq K\rbrace$. Let $K>>1$, and suppose $|E|:=|E_{K}|=\frac{1}{K}$. We will show that if 
$
K\le e^{\eps_0\sqrt{\log N}}$ ($\eps_0$ is a small absolute positive number), this would bring us the contradiction. Therefore, we will get an estimate from above on the measure of the set $E$ of ``bad" directions.

Let $N=\exp\bigg(\frac{\log K}{\eps_0}\bigg)^2$ (whenever an integer is defined to be a non-integer, it is understood that one rounds). Then $\forall\theta\in E_{K}$,
$$K\geq ||f_{N,\theta}||_{L^\infty(x)}\int{f_{N,\theta}(x)dx}\geq ||f_{N,\theta}||^2_{L^2(x)}\approx ||\widehat{f_{N,\theta}}||^2_{L^2(y)}\geq C\int_1^{3^{N/2}}{|\widehat{\nu}(y)|^2dy}$$
(Note that $\frac{1}{2}3^N\chi_{[-3^{-N},3^{-N}]}$ converges quickly as an approximate unit)

Splitting $[1,3^{N/2}]$ into $N/2$ pieces $[3^k,3^{k+1}]$ and taking blocks of such consecutive pieces, the blocks cannot all have large intergrals simultaneously. That is, if we fix $0<A'<B'<1/2$, then $\forall m\in [0,A'N]  \exists \,n\in [B'N,N/2]$ s.t. $$
\frac{1}{|E_{K}|}\int_{E_{K}}{\int_{3^{n-m}}^{3^n}{|\widehat{\nu_N}|^2dyd\theta}}\leq CKm/N\,.
$$
 So if 
 $$
 E:=\lbrace\theta\in E_{K}: \int_{3^{n-m}}^{3^n}{|\widehat{\nu_N}|^2dyd\theta}\leq 2CKm/N\rbrace\,.
 $$
  then $|E|\geq\frac{1}{2K}$.

Define $c_1=\cos(\theta-\pi/2)$, $c_2=\cos(\theta-7\pi/6)$, $c_3=\cos(\theta+\pi/6)$, and similarly, $s_1=\sin(\theta-\pi/2)$, etc. Let $$\phi_\theta(y)=\frac{1}{3}\sum_{j=1}^3{e^{-ic_jy}}.$$ Then $\widehat{\nu_N}(y)=\prod_{k=0}^N{\phi_\theta(3^{-k}y)}\approx\prod_{k=0}^n{\phi_\theta(3^{-k}y)}$ for $y\in [3^{n-m},3^n]$. So changing variable ($y\to 3^ny$) and reindexing the product ($k\to n-k$), we get

$$\int_{3^{n-m}}^{3^n}{|\widehat{\nu_N}|^2dyd\theta}\approx\int_{3^{n-m}}^{3^n}{\prod_{k=0}^n{|\phi_\theta(3^{-k}y)|^2}dyd\theta}=3^n\int_{3^{-m}}^1{\prod_{k=0}^n{|\phi_\theta(3^{k}y)|^2}dyd\theta}$$

So for $\theta\in E$, $3^n\int_{3^{-m}}^1{\prod_{k=0}^n{|\phi_\theta(3^{k}y)|^2}dyd\theta}\leq \frac{CKm}{N}$. Later, we will let $m\approx \log\,K$ and $l=\alpha \log\,K$ (for an appropriate $\alpha$) and show that $\exists\theta\in E$ such that 

\begin{equation}\label{totalest}
\int_{3^{-m}}^1{\prod_{k=0}^n{|\phi_\theta(3^{k}y)|^2}dy}\geq C3^{-n+m-A\,\ell\,m},
\end{equation}
 resulting in a choice of $m$.

First, let us write $\prod_{k=0}^n\phi_\theta(3^ky)=P_\theta(y)=P_{1,\theta}(y)P_{2,\theta}(y)$, where $$P_{1,\theta}(y)=\prod^{m}_{k=0}\phi_\theta(3^ky)\,\,\,\text{and}\,\,\, P_{2,\theta}(y)=\prod_{k=m+1}^n\phi_\theta(3^ky).$$ We want
\begin{equation}\label{P2below}
\int_{3^{-m}}^1{|P_{2,\theta}|^2dy}\geq C3^{m-n}
\end{equation}
with a proportion of the contribution to the integral separated away from the complex zeroes of $P_{1,\theta}$.

First, Salem's trick for $\int_0^1{|P_{2,\theta}(y)|^2dy}$:

Let $h(y):=(1-|y|)\chi_{[-1,1]}(y)$, and note that $\hat{h}(\lambda)=C\frac{1-\cos\lambda}{\lambda ^2}>0$. Then if we write $P_{2,\theta}=3^{m-n}\sum_{j=1}^{3^{n-m}}{e^{i\lambda_jy}}$, we get $$\int_0^1{|P_{2,\theta}(y)|^2dy}\geq 2\int_{-1}^1{h(y)|P_{2,\theta}|^2dy}\approx (3^{m-n})^2[3^{n-m}+\sum_{j\neq k, j,k=1}^{3^{n-m}}{\hat{h}(\lambda_j-\lambda_k)}]\geq 3^{m-n}.$$

To show that this is not concentrated on $[0,3^{-m}]$, we will use Lemma \ref{CET}. We get $$\int_0^{3^{-m}}{|P_{2,\theta}(y)|^2dy}=3^{-m}\int_0^1{|P_{2,\theta}(3^{-m}y)|^2dy}=3^{-m}(3^{m-n})^2\int_0^1{|\sum_{j=1}^{3^{n-m}}{e^{i\lambda_j3^{-m}y}}|^2dy}.$$

Note that in this expression, the frequencies $\beta_j:=3^{-m}\lambda_j$ are the frequencies of $\widehat{f_{n-m,\theta}}$, but that they have been subjected to two changes of variables acting on $y$ by a cumulative factor of $3^{n-m}$. By the definition of $E_{K}$, $\beta_j$ can belong to a fixed unit interval for at most $K$ values of $j$. So the lemma tells us that $$\int_0^{3^{-m}}{|P_{2,\theta}|^2dy}\leq C3^{-m}(3^{m-n})^2K3^{n-m}\leq C\frac{3^{-m}K}{3^{n-m}}\leq\frac{1}{2}\int_0^1{|P_{2,\theta}(y)|^2dy},$$ if we introduce the assumption $3^m=CK$ for $C$ large enough. So now we have $\eqref{P2below}$.

\subsection{The estimate of $P_{2,\theta}$ on the set of small values of $P_{1,\theta}$}
\label{P2}

To get $\ref{totalest}$ from $\ref{P2below}$, we will show that a proportion of $\ref{P2below}$ must have come from outside of the $\textbf{set of small values}$ $SSV(\theta, \ell)$ of $P_{1,\theta}$, so that in $\ref{P2below}$, we may restrict the integration domain to the complement of $SSV(\theta,\ell)$ and bound $P_1$ by $3^{-A\ell m}$ from below.

Let $\ell= C_0\, m$, where the large absolute $C_0$ will be chosen later.

\noindent{\bf Definition.}
$$
SSV(\theta, \ell):=\{y\in [0,1]: |P_{1,\theta}(y)|< 3^{-A\,\ell\,m}\}\,,
$$
where $A$ is another large absolute constant to be seen in Section $\ref{complex}$.

This is the desired inequality:
\begin{equation}\label{final3}
\frac{1}{|E|}\int_E \int_{[3^{-m},1]\cap SSV(\theta,\ell)}|P_{2,\theta}(y)|^2\,dy\,d\theta\leq \frac{\varepsilon_0}{3^{n-m}}\leq 0.5\,\frac{1}{|E|}\int_E \int_{[3^{-m},1]}{|P_{2,\theta}(y)|^2\,dy\,d\theta},
\end{equation}
because it gives us
$$\frac{1}{|E|}\int_E \int_{[3^{-m},1]\setminus SSV(\theta,\ell)}|P_{2,\theta}(y)|^2\,dy\,d\theta\ge 0.5\,\frac{1}{|E|}\int_E \int_{[3^{-m},1]}|P_{2,\theta}(y)|^2\,dy\,d\theta\ge a\,\frac1{3^{n-m}}\,.$$

To get it, we will need to split $P_{2,\theta}(y)$ into two parts, $\Pshp$ and $\Pflt$, because Lemma $\ref{CET}$ applied to $\Pshp$ will get us part of the way there, and the claims of Section $\ref{complex}$ applied to $\Pflt$ will finish the estimate.

Introduce 
$$
\Pflt := \prod_{k=m+1}^{m+\ell/2}\phi_\theta(3^k y)\,,\, \Pshp:= \prod_{k=m+\ell/2+1}^n\phi_\theta(3^k y)\,.
$$
Also 
$$
R(x):=\prod_{k=m+1}^{m+\ell/2}\frac{7+ 2\cos (3^k x)}{9}\,.
$$

We will prove now
\begin{lemma}
\label{R1}
$|P_{2,\theta}^{\flat}(y)|^2 \le \min (R(y(c_2(\theta)-c_1(\theta))), R(y(c_3(\theta)-c_1(\theta))))\,.$
\end{lemma}

\begin{proof}
Notice that
$$
|2\cos \alpha + e^{i\beta}|^2 \le 7 + 2\cos(\alpha-\beta)
\,.
$$
In fact, 
$$
|2\cos \alpha + e^{i\beta}|^2=
 4 \cos^2\alpha + 1 + 4\cos\alpha\cos\beta \le 5 + 2\cos(\alpha +\beta) + 2 \cos(\alpha-\beta)\le
 $$
 $$
 7 + 2\cos(\alpha-\beta)\,.
 $$
 We can write
 $$
 |e^{ic_1 y} + e^{ic_2y} +e^{ic_3y}|^2 = |1 + e^{i(c_2-c_1)y} +e^{i(c_3-c_1)y}|^2 =:|1 + e^{it_1} +e^{it_2}|^2\,,
$$
where $t_1 :=(c_2-c_1)y, t_2:=(c_3-c_1)y$.
This is the same as $|e^{-i\frac{t_1+t_2}{2}} +2\cos (\frac{t_1-t_2}{2})|^2$. We use $\alpha=\frac{t_1-t_2}{2}, \beta=-\frac{t_1+t_2}{2}$, then $\alpha-\beta= t_1= (c_2-c_1)y$. By symmetry we could have used $\alpha+\beta=-t_2=
-(c_3-c_1)y$.

\end{proof}

For $y$ in $SSV$, we want $\Pflt$ to depend only on $\theta$, so that we may integrate it independently of $y$. In Section $\ref{complex},$ we will see that $SSV$ is contained in a union of small neighborhoods of the complex zeroes of $P_{1,\theta}$, and that the zeroes are in fact simple, depending differentiably on $\theta$. So we divide $SSV$ into the intersections of the neighborhoods of these zeroes with the real interval $[3^{-m},1]$. Lemma $\ref{R2}$ says that within one such interval, our Riesz products estimates on $\Pflt$ are absolutely comparable independent of $y$.

In the following, $j(s,k,t,\theta)$ will be a small interval containing $y_s(k,t,\theta)$. Point $y_s(k,t,\theta)$ will be  the real part of a complex zero of $P_{1,\theta}$. At any rate, the intervals $j(s,k,t,\theta)$ union over $s,k,t$ to give $SSV(\theta,\ell)$. (Roughly, $s$ tells us which factor of $P_1$ was zero, $k$ tells us which zero, and $t$ tells us which subinterval of length $3^{-Bm}$ $\theta$ belongs to. See Section \ref{complex} for details; one and the same zero of $\phi_\theta$ generates several intervals of smamllness of $P_1$.)

Roughly, 
$$
3^{-s}g_1(k,t,\theta)=y_s(k,t,\theta)(c_2(\theta)-c_1(\theta))\,,
$$
 and 
 $$
 3^{-s}g_2(k,t,\theta)=y_s(k,t,\theta)(c_3(\theta)-c_1(\theta))\,.
 $$

Later, we will see that there are a few pathological directions $\theta\in W$, which we will isolate in a small neighborhood $W_m$:

$$W:= \{\pi/2, 5\pi/6, \pi/6\}.\text{ Let }W_m:= \cup_{w\in W} B(w, 2\cdot 3^{-40m}).$$

Now we are going to estimate
$$
I:=\frac1{|E|}\int_E\int_{[0,1]\cap SSV(\theta,\ell)} |\Pflt|^2|\Pshp|^2\,dy\,d\theta\,.
$$
Recalling that $\frac1{|E|}=2K$ we have
$$
I \le 2K\int_{2W_m}\int_{[0,1]\cap SSV(\theta,\ell)}|\Pshp(y)|^2\,dy\,d\theta+ 
$$
$$
C_4\,K \sum_{t=1}^T\int_{J^t\cap E}|\Pflt|^2\int_{[0,1]\cap SSV(\theta,\ell)}|\Pshp|^2\,dy\,d\theta=: I_W +\sum_t I_t\,.
$$
Here we use the notations after Lemma \ref{SSVle}. We need to know now that $T\le 3^{Bm}$.
Now, $SSV(\theta,\ell) \subset \cup_{k=1}^{K(t)}\cup_{s=\kappa}^m j(s,k,t,\theta)$ (Lemma \ref{SSV3}).

Let us first estimate 
$$
\int_{j(s,k,t,\theta)} |\Pshp|^2 dy\,,\,\,\text{having in mind that}\,\, \theta\in E\,.
$$

Using that $\theta \in E$ and applying the Lemma \ref{CET} we get ($|c_j|=1$ appears because of the change of variable $[0,1]\rightarrow j(s,k,t,\theta)$, $ y= 3^{-s-\ell}x +b$, $c_j= e^{i\lambda_j' b}$)
$$
\int_{j(s,k,t,\theta)}{|\Pshp|^2dy}=(3^{m+\ell/2-n})^2\int_{j(s,k,t,\theta)}{|\sum_{j=1}^{3^{n-m-\ell/2}}{e^{i\lambda_j'y}}|^2dy}\leq
$$
$$ 
C(3^{m+\ell/2-n})^23^{-s-\ell}\int_0^1{|\sum_{j=1}^{3^{n-m-\ell/2}}{c_j\,e^{i\widetilde{\lambda_j'}x}}|^2dx}
$$
$$
\leq C\frac{3^{-m-l/2}}{(3^{n-m-l/2})^2}K\cdot 3^{n-m-l/2}\leq \frac{CK3^{-m}}{3^{n-m}}\,,
$$
because in each {\bf unit} interval we have at most $K\, (3^{m+\ell/2}3^{-s-\ell})^{-1}$ frequencies $\widetilde{\lambda_j'}$.

Therefore, recalling that number of intervals is bounded by constant times $3^m$ we get for every $\theta\in E$
$$
\int_{SSV(\theta,\ell)} |\Pshp|^2 dy \le \frac{C_7\,K}{3^{n-m}}\,.
$$
Using the fact that $|2W_m|\le C_8 e^{-40m}$ we estimate
$$
I_W = o(\frac1{3^{n-m}})\,,\,\text{so}\,\, I_W= o(\int_{[3^{-m},1]} |P_{2,\theta}(y)|^2\,dy)\,.
$$

Recall Lemma \ref{SSVle} and the notion of $\kappa$ after it to formulate
\begin{lemma}
\label{SSV31}
If $\theta\in J^t$ then $SSV(\theta,\ell) \subset \cup_{k=1}^{K(t)}\cup_{s=\kappa}^m j(s,k,t,\theta)\,.$
\end{lemma}

Let us now estimate $I_t =C\,K\int_{J^t\cap E}\int_{SSV(\theta,\ell)}...$. When we fix $t$ we use Theorem \ref{conclusion}. Then either on $J^t$ as a whole, or on one of the subdivision intervals $J^t_u, u=1,...,U\le B_0\,m$ we have for each fixed $k=1,...K(t)$
\begin{equation}
\label{determinant}
|g_1'(k,t,\theta|\ge a_2\,\delta_0
\end{equation} or the same happens with $g_2(k,t,\theta)$ on the whole $J^t_u$. 

Then again exactly as before, by using $\theta\in E$ and the Carleson Embedding Theorem, we get
$$
\int_{SSV(\theta,\ell)} |\Pshp|^2 dy \le \frac{C_7\,K}{3^{n-m}}\,.
$$

 Now,
$$
\int_{J^t_u\cap E}\int_{j(s,k,t,\theta)}|P^{\flat}_{2,\theta}(y)|^2|\Pshp|^2\,dy\,d\theta\le 
$$
$$
\int_{J^t_u\cap E}R(y_s(k,t,\theta)(c_2(\theta)-c_1(\theta)))|\int_{j(s,k,t,\theta)}\Pshp|^2\,dy\,d\theta \,,\, s=\kappa,...,m\,.
$$
Here we are using Lemmas \ref{R1} and \ref{R2}.

And now we need to estimate only
$$
\int_{J^t_u}R(3^{-s}g_1(k,t,\theta))\,d\theta\,.
$$
Notice that we we throw away $\cap E$ at this stage.

We change the variable $v=3^{-s}g_1(k,t,\theta), \theta\in J^t_u$, and notice that this is a monotone change of variable and (see Theorem \ref{conclusion}) 
$$
\bigg|\frac{\partial v}{\partial\theta}(\theta)\bigg|\ge a_2\delta_0\cdot 3^{-m}\,.
$$

Then a Riesz product observation shows
$$
\int_{J^t_u}R(3^{-s}g_1(k,t,\theta))\,d\theta \le (a_2\delta_0\cdot 3^{-m})^{-1}\int_0^{2\pi} \prod_{k=m+1}^{m+\ell/2}\frac{7+2\cos (3^k\,v)}{9} dv \le
C\, 3^m\,\bigg(\frac79\bigg)^{\ell/2}\,.
$$

We already proved ($\theta\in E\cap J^t_u$)
$$
\int_{SSV(\theta,\ell)} |\Pshp|^2 dy \le \frac{C_7\,K}{3^{n-m}}\,.
$$

Therefore,
\begin{equation}
\label{Jt}
I_t \le 3^m\,K\,U\,\bigg(\frac79\bigg)^{\ell/2}\,\frac{C_7\,K}{3^{n-m}}\,,
\end{equation}

Gathering the estimate $U\le B_0\,m$, the estimate for $I_W$, and the estimate \eqref{Jt} together we obtain
by  recalling that $K=3^m/R$, where $R$ is a large absolute constant:
$$
I:=\frac1{|E|}\int_E\int_{[0,1]\cap SSV(\theta,\ell)} |\Pflt|^2|\Pshp|^2\,dy\,d\theta \le 
$$
\begin{equation}
\label{final1}
3^m\,K(K 3^{-40m} +
B_0\,3^{B\,m}\,m\,\bigg(\frac79\bigg)^{\ell/2})\frac1{3^{n-m}}\,.
\end{equation}
If we choose
$$
\ell= 100\,B\,m\,,
$$ we get from \eqref{final1} that for every $\theta$

\begin{equation}
\label{final2}
I=o(\frac1{3^{n-m}})\,,\,I= o(\int_{[3^{-m}, 1]} |\Pflt|^2|\Pshp|^2\,dy\,.
\end{equation}

Therefore
$$\frac{1}{|E|}\int_E \int_{[0,1]\setminus SSV(\theta,\ell)}|P_{2,\theta}(y)|^2\,dy\,d\theta\ge 0.5\,\frac{1}{|E|}\int_E \int_{[3^{-m},1]}|P_{2,\theta}(y)|^2\,dy\,d\theta\ge a\,\frac1{3^{n-m}}\,.$$

On the other hand on $[0,1]\setminus SSV(\theta,\ell)$ we have $|\Pa|\ge 3^{-A\,\ell\, m}= 3^{-100AB\, m^2}$
Now we can write
$$
C\frac{Km}{N}\geq 3^n\int_{3^{-m}}^1{|\Pa|^2|\Pb|^2dy}\geq C'3^n\,3^{-100AB\, m^2}\,3^{m-n}\geq C'\frac{3^{m}}{3^{100AB\, m^2}},
$$ 
i.e., 
$$
m\geq C''\frac{N}{3^{100AB\, m^2}}
$$ 
which implies the contradiction if we choose $m=\eps_0\sqrt{\log N}$ with sufficiently small $\eps_0$, for example,
$\eps_0 = \frac12 \sqrt{\frac1{100AB}}$.

Therefore, 
$$
K\le R\,e^{\eps_0\sqrt{\log N}}
$$
brings the contradiction to
$$
1/|E| = 2K\,,
$$ and, hence, 
\begin{equation}
\label{main2}
|E|= 1/2K \le C\,e^{-2\eps_0\sqrt{\log N}}\,.
\end{equation}

\bigskip

Recall that $E$ was the set of singular directions, on which we do not have overlap of $K$ or larger number of triangles of size $\le 3^{-N}$. Any other direction $\theta$ will be in the set of good directions. Such an overlap will happen and (see  Section \ref{combinatorics})

\begin{equation}
\label{gooddir}
|\mathcal{L}_{\theta, N\,3^N}|\le \frac{C}{K}\,.
\end{equation}

So we proved

\begin{theorem}
\label{mainth}
$$
\int_0^{\pi}|\mathcal{L}_{\theta, n}|\,d\theta \le C\, e^{-c\,\sqrt{\log \log n}}\le \frac{C}{\log\log n}\,.
$$
\end{theorem}

\section{Combinatorial part}
\label{combinatorics}

To prove \eqref{gooddir} let us assume that a projection on direction $\theta$ has the overlap of $Q_1,..., Q_K$, where these $Q_j$ are triangles of size $3^{-n}, n\le N$. We cal the the collection of all $3^{-s}$ tringles by $T_s$.
If we project $T_n$ we get ($C\ge 1$ is an absolute constant)
$$
|proj \,T_n| \le C\,3^{-n} + 3^{-n} (3^n-K)\,.
$$

The first term estimates $|proj (Q_1\cup...\cup Q_K)|$ and the second term is just the sum of projections of the $Q_{K+1},...,Q_{3^n}$. The estimate is not impressive, it is close to $1$, not to $0$, but let us iterate it. Inside
each of $Q_{K+1},...,Q_{3^n}$ there is a stack of tringles $Q_{j1},...,Q_{jK}$ which are $3^{-n}$-dilated copies of $Q_1,..., Q_K$ and which will overlap when projected on $\theta$-direction.  Therefore the second term can be improved: it becomes $(C\,3^{-2n} + 3^{-2n} (3^n-K))(3^n-K)$. So we get
\begin{equation}
\label{iteration1}
|proj\, T_{2n}| \le C\,3^{-n} + C\,3^{-2n}(3^n-K) + 3^{-2n} (3^n-K)^2\,.
\end{equation}
Next iteration of the same self-similarity observation gives

\begin{equation}
\label{iteration2}
|proj \,T_{3n}| \le C\,3^{-n} + C\,3^{-2n}(3^n-K)+ C\,3^{-3n}(3^n-K)^2 + 3^{-3n} (3^n-K)^3\,.
\end{equation}

After $X$ iterations

\begin{equation}
\label{iteration3}
|proj \,T_{Xn}| \le C\,\sum_{k=1}^X 3^{-kn}(3^n-K)^{k-1} + 3^{-Xn} (3^n-K)^X \le \frac{C\, 3^{-n}}{1-(1-K \,3^{-n})} + e^{-K3^{-n}X}\,.
\end{equation}
Therefore, if $X\ge 3^n \frac{\log K}{K}$ (for example, $X=3^n$)

\begin{equation}
\label{iteration4}
|proj \,T_{Xn}| \le  \frac{2C}{K}\,.
\end{equation}

But then
$$
\mathcal{L}_{\theta, N3^N} \le \mathcal{L}_{\theta, n3^n}  = |proj \,T_{n3^n}| \le  \frac{2C}{K}\,,
$$
and \eqref{gooddir} is proved.

\section{The complex analytic part}
\label{complex}

\subsection{Zeros of $\varphi_\theta(z)$}
\label{entire}

In this section (up to factor $3$ from before) $\varphi_\theta(z) := e^{-ic_1z} + e^{-ic_2z} + e^{-ic_3z}$, where $c_1=\cos (\theta-\frac\pi{2})\,,\,c_2=\cos (\theta-\frac{7\pi}{6})\,,\,c_3=\cos (\theta+\frac{\pi}{6})$. We need to know how the zeros are separated and how they behave with changing of $\theta\in [0,2\pi)$.

Notice that there are three sectors $S_1, S_2, S_3$ such that, say, $c_1\ge a$ ($a$ is an absolute positive constant) in $S_1$ and $c_2,c_3 <0$ in $S_1$ (and similarly for other sectors). Sectors have apperture $\pi/3$ each, and are symmetric with respect to rays $\pi/2$, $7\pi/6$, and $-\pi/6$ correspondingly. If, say, $e^{i\theta}\in S_1$ we get that for $z=x+iy$ with $y\ge H=H(a)$, $|\varphi_\theta(z)|\ge 1$. The same for other sectors, so always  if $\theta\in S_1\cup S_2\cup S_3$ we have
\begin{equation}
\label{H}
|\varphi_\theta(x+ iH)|\ge 1\,.
\end{equation}
If we happen to be in $e^{i\theta}\in - S_1$ then $c_2,c_3 \ge 0$, and $c_1<-a$. Then, 
\begin{equation}
\label{H-}
|\varphi_\theta(x- iH)|\ge 1\,.
\end{equation}
Similarly, we could have reasoned that $\varphi_\theta(-z) =\varphi_{\theta+\pi} (z)$. Note also that $|\varphi|\leq C(H)$ when $\Im(z)\leq H$ (where $C(H)$ is a constant depending on $H$).

Every rectangle $B:=[x_0-1,x_0+1]\times [-H,H]$ will be called a box. Because of Lemmas $\ref{schke1}$, $\ref{schke2}$, and $\ref{schke3}$, we may say the following:

In every box we have at most absolute constant $M$ of zeros of $\varphi_\theta(z)$ uniformly in $\theta\in [0,2\pi)$.

For a certain uniform in $\theta$ absolute constant $\eta>0$ we have
\begin{equation}
\label{disks}
\{z: \Im z\in (-H,H): |\varphi_\theta (z)|<\delta\} \subset \cup_{\lambda_i} D(\lambda_i, \delta^{\eta})\,.
\end{equation}
Here $\{\lambda_i=\lambda_i(\theta)\}$ are zeros of $\varphi_\theta$.

Notice also that uniformly in $\theta$ for all sufficiently large $m$
\begin{equation}
\label{number}
|\{i: |\lambda_i|\le 3^m|\} \le C\,3^m\,.
\end{equation}

The constant $C$ is absolute and uniform in $\theta$. (This last fact could have also been obtained from the theory of entire functions with growth conditions.)

\bigskip

\subsection{Zeros of $\varphi_{\zeta}(z)$}
\label{zeta}

Now $\zeta=\theta +i\sigma$, $\varphi_{\zeta}(z) = e^{-ic_1(\zeta) z}+ e^{-ic_2(\zeta) z}+e^{-ic_3(\zeta) z}$.

Recall that $c_1(\zeta)=\cos (\zeta-\frac\pi{2})\,,\,c_2=\cos (\zeta-\frac{7\pi}{6})\,,\,c_3=\cos (\zeta+\frac{\pi}{6})$,  $c_1=\cos (\theta-\frac\pi{2})\,,\,c_2=\cos (\theta-\frac{7\pi}{6})\,,\,c_3=\cos (\theta+\frac{\pi}{6})$,and $s_1=\sin (\theta-\frac\pi{2})\,,\,s_2=\sin (\theta-\frac{7\pi}{6})\,,\,s_3=\sin (\theta+\frac{\pi}{6})$. Then we write ($\zeta=\theta+i\sigma$)
$$
\varphi_{\zeta}(z) = e^{-ic_1\ch\sigma \,(x+iy)} e^{-s_1\sh\sigma \,(x+iy)} + e^{-ic_2\ch\sigma \,(x+iy)} e^{-s_2\sh\sigma \,(x+iy)}+
$$
$$
e^{-ic_3\ch\sigma \,(x+iy)} e^{-s_3\sh\sigma \,(x+iy)}\,.
$$
Fix 
$$
Q= \{\zeta: -\pi/2<\theta< 5\pi/2, -\sigma_0<\sigma<\sigma_0\}\,.
$$
where $\sigma_0$ is a small positive absolute constant.
\begin{lemma}
\label{M}
There is an absolute constant $M$ such that for any $\zeta\in Q$ and for any box $B_{x_0}, x_0>0,$
\begin{equation}
\label{eqM}
\text{card}(\lambda\in B: \varphi_{\zeta}(\lambda) =0) \le M\,.
\end{equation}
\end{lemma}

\begin{proof}
This is again Lemma $\ref{schke1}$. We want $|\varphi|$ to have an upper bound on the boundary of a box, and a lower bound at a point inside the box.

Consider first the case of $\sigma\ge 0$. In a box $B_{x_0}$ we have the estimate from above
$$
|\varphi_{\zeta} (z) | \le C_0\,,
$$ 
if $x_0\,\sh\sigma \le 10$. If $x_0\,\sh\sigma \ge 10$ we have
$$
|\varphi_{\zeta} (z) | \le C_0\, e^{s_i\sh\sigma\, x_0}
$$ for some $i=1,2,3$.
We want to prove that the box contains a point $w$ such that
$$
|\varphi_{\zeta} (w) | \ge c_0\,,
$$ 
if $x_0\,\sh\sigma \le 10$. If $x_0\,\sh\sigma \ge 10$ we will have
$$
|\varphi_{\zeta} (w) | \ge c_0\, e^{s_i\sh\sigma\, x_0}
$$ for the same $i=1,2,3$ as above.
The case $x_0\,\sh\sigma \le 10$ is treated exactly as in \eqref{H}, \eqref{H-}, and as a result, we find $w= x_0\pm iH$ satisfying $|\varphi_{\zeta} (w) | \ge c_1$.  May be only $H$ should be chosen bigger, but its size is not dependent on anything (except number $10$).

If $x_0\,\sh\sigma \ge 10$, consider several cases. Real line is split to intervals of length $\pi$, we call such interval positive if $\sin$ on it is positive. Notice that given a positive interval only one or two of $\theta-\pi/2, \theta-7\pi/6, \theta +\pi/6$  belongs to it $mod \,2\pi$.  

Case 1 is when only one, say, $\theta-\pi/2$ belongs to a positive interval $mod \,2\pi$. Notice that then $s_1 \ge \sqrt{3}/2$, $s_2, s_3 \le 0$, $\sh\sigma \,x_0\ge 10$, and as a result 
$$
|\varphi_{\zeta} (x_0)| \ge  e^{s_1\sh\sigma\, x_0} - 2 \ge \frac12 e^{s_1\sh\sigma\, x_0}\,.
$$

Case 2 is when both elements of a pair, say $s_1, s_2$, are non-negative, $s_3 \le 0$. Consider the situation when $s_1 > s_2 +1/10$. Then
it is easy to see that $s_1\ge 1/2$, and then
$$
|\varphi_{\zeta} (x_0)| \ge  e^{s_1\sh\sigma\, x_0}( 1  -  e^{-1/10\sh\sigma\, x_0}) -1 \ge 
$$
$$
 (1-1/e)e^{s_1\sh\sigma\, x_0} -e^{-5} e^{s_1\sh\sigma\, x_0} \ge c_0 e^{s_1\sh\sigma\, x_0}\,.
$$

We are left to consider the situation when $s_1, s_2$ are non-negative, $s_3 \le 0$, and $s_2 \le s_1 \le s_2 +1/10$.
Notice that in this case  $|c_1|, |c_2| \ge 1/3$, $s_1 \ge 1/2$.  Suppose $c_1>0$, choose $h>0$ such that $e^{c_1 h} \ge 2$. It is enough to choose $h=3$. Consider $w= x_0 +ih$. If there would be $c_1 <0$ (and so $c_1<-1/3$) we would choose $w=x_0-ih$.  In both cases, $c_2$ has an opposite sign, and we can notice that $s_3 \le -\sqrt{3}/2$. Hence,

$$
|\varphi_{\zeta}(w)| \ge 2\,e^{s_1\sh\sigma\, x_0}  - e^{s_2\sh\sigma\, x_0} - e^{|c_3|\ch\sigma_0\, h} e^{-\frac{\sqrt{3}}{2}\sh\sigma\, x_0}\,.
$$
Using the facts that $\sigma_0\le 1/10$, $\sh\sigma\, x_0\ge 10$, $h=3$, $s_2\le s_1$, $s_1\ge 1/2$, we get from the above:
$$
|\varphi_{\zeta}(w)| \ge e^{s_1\sh\sigma\, x_0}-e^{6}e^{-\frac{\sqrt{3}}{2}\sh\sigma\, x_0}\ge \frac12\,e^{s_1\sh\sigma\, x_0}\,.
$$

We finished to consider the case of $\sigma\ge 0$, the case $\sigma<0$ is taken care of in a symmetric fashion.

\bigskip

To finish the estimate \eqref{eqM} we notice that in a  doubled box $2\,B_{x_0}$ we have the estimate from above
\begin{equation}
\label{above1}
|\varphi_{\zeta} (z) | \le C_0\,,
\end{equation} 
if $x_0\,\sh\sigma \le 10$. If $x_0\,\sh\sigma \ge 10$ we have
\begin{equation}
\label{above2}
|\varphi_{\zeta} (z) | \le C_0\, e^{s_i\sh\sigma\, x_0}
\end{equation} for  $i=1,2,3$ for which $s_i$ was the largest.
We have already proved that the box $B_{x_0}$ with sufficiently large absolute $H\ge 3$ contains a point $w$ such that
\begin{equation}
\label{below1}
|\varphi_{\zeta} (w) | \ge c_0\,,
\end{equation}
if $x_0\,\sh\sigma \le 10$. If $x_0\,\sh\sigma \ge 10$ we proved the existence of $w\in B_{x_0}$ such that
\begin{equation}
\label{below2}
|\varphi_{\zeta} (w) | \ge c_0\, e^{s_i\sh\sigma\, x_0}
\end{equation} for the same $i=1,2,3$. Here $c_0, C_0$ are absolute constants.

Lemma $\ref{schke1}$ now applies to all of our cases.
\end{proof}

 \bigskip
 
\subsection{The set of small values of $\Pa$}
\label{small values}

In this section we want to investigate the set
$$
\mathcal{G}:=\{y\in [3^{-m}, 1]: |\Pa(y)|=|\varphi_{\theta}(y)\varphi_{\theta}(3y)\cdot...\cdot\varphi_{\theta}(3^my)| \ge 3^{-A\,m\ell}\}\,.
$$
If $\Omega (k, \theta, 3^{-A\ell}):= \{y\in [3^{-m}, 1]: |\varphi_{\theta}(3^ky)| < 3^{-A\ell}\}$, then
the set of small values 
$$
\Omega (\theta, \ell) = [3^{-m}, 1]\setminus \mathcal{G} \subset\cup_{k=0}^m \Omega (k, \theta, 3^{-A\ell})\,.
$$

We already saw that if $A\ge 1/\eta$

$$
\{ z=x+iy, 0<x< 3^m, -H<y<H: |\varphi_{\theta}(z)| < 3^{-A\ell} \subset \cup_i D(\lambda_i(\theta), 3^{-\ell})\}\,,
$$
where $\lambda_i(\theta)$ are zeros of $\varphi_{\theta}(z)$. In proving this we essentially used only the absolute bound on the number od zeros of $\varphi_{\theta}$ in the box, and the fact that each box has a point where $|\varphi_{\theta}(w)|$ is comparable with the $\max |\varphi_{\theta}|$ over the box.

But the same is formulated for $\varphi_{\zeta}(z)$, $\zeta\in Q$, in Lemma \ref{M} and in \eqref{above1}, \eqref{below1}, \eqref{above2}, \eqref{below2},
 so we trapped the set of small values of $\varphi_{\zeta}$ in the collection of discs:
 $$
\{ z=x+iy, 0<x< 3^m, -H<y<H: |\varphi_{\zeta}(z)| < 3^{-A\ell} \subset \cup_i D(\lambda_i(\zeta), 3^{-\ell})\}\,,
$$
where $\lambda_i(\zeta)$ are zeros of $\varphi_{\zeta}(z)$.

\bigskip

Now we want to show that $\Pb$ is still large enough away from the set where $\Pa$ is small. Let $\OTL :=\lbrace y\in [3^{-m},1]:|\Pa|\leq C3^{-A\,\ell\,m}\rbrace$. Then $\OTL$ is contained  in  contractions by the factors $3^{-k}$, $k=1,2,...,m$, of  $3^{-\ell}$-neighborhood of the complex zeroes of $\phi_\theta(z)$. By \eqref{number} and by this whole subsection $\OTL\subset\bigcup_{j=1}^L J_j(\theta)$, where $L\leq C\,3^m$ and $|J_j(\theta)|\leq C\,3^{-k-\ell}$ ($k=k(j)\in \{1,2,...,m\}$).

\subsection{Branch points of $\phi_\theta$}
\label{branch1}

For a point $\zeta=\theta + i\sigma\in Q$ we call $z$ a {\it branch point} of $\phi_\zeta(z)$ if  $z\in [-3^m-1, 3^m+1]\times [-H/2,H/2]$ is such that
\begin{equation}
\label{eqbranch}
\begin{cases}
\phi_\zeta(z) =0 \\ 
\frac{\partial}{\partial z} \phi_\zeta(z) =0\,.
\end{cases}
\end{equation}

\begin{lemma}
\label{realbrp}
For real $\zeta=\theta\in Q\cap\R$ there are no branch points.
\end{lemma}

\begin{proof}
We will prove more: that for $\zeta=\theta\in Q\cap\R$ the system \eqref{eqbranch} has no solutions $z$ at all.

As always $c_1=\cos(\theta-\pi/2), c_2=\cos(\theta-7\pi/6), c_3=\cos(\theta+\pi/6)$, also
$$
b:= (c_2-c_1)/ \sqrt{3}=\sin(\theta-5\pi/6), a:= (c_1-c_3)/\sqrt{3} =\sin (\theta-\pi/6), -(a+b) = (c_3-c_2)/ \sqrt{3}=\cos \theta\,.
$$
If \eqref{eqbranch} is valid then
\begin{equation}
\label{eqbranch1}
\begin{cases}
e^{iZb} + e^{-iZa} =-1\\ 
be^{iZb} -ae^{-iZa} =0\,.
\end{cases}
\end{equation}
Hence
\begin{equation}
\label{eqbranch2}
\begin{cases}
e^{iZb} =-\frac{a}{a+b}\\ 
e^{-iZa} =-\frac{b}{a+b}
\end{cases}
\end{equation}
has a solution ($Z=\sqrt{3}z$) . Here $a, b\in\R$. If $a=0$ or $b=0$ or $a+b=0$ there is no solution of \eqref{eqbranch2} just because the exponent cannot be zero or infinity. 

Suppose all these three numbers do not vanish. Take absolute values in \eqref{branch2}:
\begin{equation}
\label{eqbranch3}
\begin{cases}
e^{-bY} =\frac{|a|}{|a+b|}\\ 
e^{aY} =\frac{|b|}{|a+b|}\\
e^{-(a+b)Y} =\frac{|a|}{|b|}\,.
\end{cases}
\end{equation}

Consider the cases:

1. $ a>0, b>0$. Then the first gives $bY>0$, and $b>0$, so $Y>0$.
The second gives $aY<0$,  and $a>0$, so $Y<0$. Contradiction.

2. $ a<0, b<0$. Then the first gives $bY>0$, and $b<0$, so $Y<0$.
The second gives $aY<0$,  and $a<0$, so $Y>0$. Contradiction.

3. $ a>0, b<0, a+b>0$. Then $|a|>|b|$. Then the first gives $-bY>0$, and $b<0$, so $Y>0$.
The third gives $-(a+b)Y>0$,  and $a+b>0$, so $Y<0$. Contradiction.

4. $ a>0, b<0, a+b<0$. Then $|a|<|b|$. Then the second gives $aY>0$, and $a<0$, so $Y<0$.
The third gives $-(a+b)Y<0$,  and $a+b<0$, so $Y<0$. Contradiction.

5. $ a<0, b>0, a+b>0$. Then $|b|>|a|$. Then the second gives $aY>0$, and $b<0$, so $Y>0$.
The third gives $-(a+b)Y>0$,  and $a+b>0$, so $Y>0$. Contradiction.

6. $ a<0, b>0, a+b<0$. Then $|a|>|b|$. Then the first gives $-bY>0$, and $b>0$, so $Y<0$.
The third gives $-(a+b)Y>0$,  and $a+b<0$, so $Y>0$. Contradiction.

\end{proof}

Actually we just proved a little bit more. To formulate it we need some notations.
Let $W:= \{\pi/2, 5\pi/6, \pi/6\,\, (mod \,\pi)\}\cap Q$. It is a finite set, let $W_m:= \cup_{w\in W} D(w, 2\cdot 3^{-40m})$.

Recall the above definition of $a(\theta), b(\theta)$ and put ($Z=\sqrt{3}z$, $Z=X+iY$)
$$
d_{\theta}(z):= \max\bigg(\bigg| |e^{iZb(\theta)}\cdot \frac{|a(\theta)+b(\theta)|}{|a(\theta)|}-1\bigg|, \bigg| |e^{-iZa(\theta)}\cdot \frac{|a(\theta)+b(\theta)|}{|b(\theta)|}-1\bigg|\bigg)\,.
$$ 
$$
d_{\theta}:=\inf_{z\in \C} d_{\theta}(z)\,.
$$
This is called discrepancy. We actually proved the following estimate for the discrepancy.

\begin{lemma}
\label{realdiscr}
There is an absolute positive constant $c$ such that
$\min_{\theta\in Q\setminus W_m} d_{\theta} \ge c\, 3^{-40m}$.
\end{lemma}

Similarly for complex $\zeta=\theta+i\sigma$ we  have 
$$
b:= \sin(\theta+i\sigma-5\pi/6), a:= \sin (\theta+i\sigma-\pi/6)\,.
$$
\begin{equation}
\label{Imb}
 \Im b= \cos(\theta-5\pi/6)\cdot\sinh \sigma\,,\,\,\Im a= \cos (\theta -\pi/6)\cdot\sinh\sigma\,.
 \end{equation}
 
We introduce
$$
d_{\zeta}(z):= \max\bigg(\bigg| |e^{iZb(\zeta)}\cdot \frac{|a(\zeta)+b(\zeta)|}{|a(\zeta)|}-1\bigg|, \bigg| |e^{-iZa(\zeta)}\cdot \frac{|a(\zeta)+b(\zeta)|}{|b(\zeta)|}-1\bigg|\bigg)\,.
$$ 
$$
d_{\zeta}:=\inf_{z\in [-3^m-1, 3^m+1]\times [-H/2, H/2]} d_{\zeta}(z)\,.
$$

\subsection{Branch points of $\phi_\zeta$}
\label{branch2}

If we leave the real axis and venture $\zeta=\theta+i\sigma$ into a complex domain we get the factor
$e^{\pm \cos(\theta-5\pi/6)\,X\,\sinh\sigma}$ into $|e^{iZb}|$ and the factor $e^{\pm \cos(\theta-\pi/6)\,X\,\sinh\sigma}$ into $|e^{-iZa}|$. This is from \eqref{Imb}. Clearly it is very close to $1$ if $|X|\le 3^m +1$ and $|\sigma|\le 3^{-100m}$. The change ratios $|a(\theta)|/|a(\theta+i\sigma)|, |b(\theta)|/|b(\theta+i\sigma)|$ will also be very close to $1$ if $\theta, \theta+i\sigma\in Q\setminus W_m$, and $|\sigma|\le 3^{-100m}$. They differ from $1$ by at most $C\, 3^{-96m}$. Therefore we proved

\begin{lemma}
\label{compdiscr}
Let $\zeta\in Q\setminus W_m, |\Im\zeta|\le 3^{-100m}$. Then $d_{\zeta}\ge \frac{c}{2} 3^{-40m}$, where $c$ is the absolute constant from Lemma \ref{realdiscr}.
\end{lemma}

\subsection{Zeros of $\phi_\zeta$ as analytic functions: $\Lambda:\zeta\rightarrow\Lambda(\zeta)$}
\label{brzeros}

Let us fix a point $\theta_0\in Q\setminus 2\, W_m$, and consider the disc $D(\theta_0):=D(\theta_0, 2\cdot 3^{-Bm})$, where $B\ge 100$ and will be chosen later. For any $\zeta \in \bar D(\theta_0, 2\cdot 3^{-Bm})$ (the closure of the disc) we consider
zeros $Z_{m, H/4}(\zeta):=\{\lambda_i(\zeta), i=1,...,I(\zeta)\}$ of $\phi_\zeta$  lying in $[-3^m,3^m]\times [-H/4, H/4]$.
We know that there exists an absolute constant $M$ independent of $\theta_0, \zeta$ such that
\begin{equation}
\label{M2le}
\text{card} (Z_{m, H}\cap [x-1,x+1]\times [-3/2H,3/2H]) \le M\,,\, \forall x\in [-3^m,3^m]\,.
\end{equation}

\begin{lemma}
\label{cont1}
These are continuous functions on  $\bar D(\theta_0, 2\cdot 3^{-Bm})$.
\end{lemma}

\begin{proof}
Let $\zeta\in \bar D(\theta_0, 2\cdot 3^{-Bm})$. All points in $Z_{m,H}(\zeta)$ are simple zeros, this follows from Lemma \ref{compdiscr}, for example. Let $\eta(\zeta):= \min_{i\neq j, i,j\le I(\zeta)}|\lambda_i(\zeta)-\lambda_j(\zeta)|$.
Then $\eta>0$. Fix $i$, call $\lambda=\lambda_i(\zeta)$.
We know that there exists an absolute constant $M$ independent of $\theta_0, \zeta$ such that
\begin{equation}
\label{M2}
\text{card}(Z_{m, H}(\zeta)\cap [\Re\lambda -1, \Re\lambda+1]\times [-H, H]) \le M\,.
\end{equation}

Consider the discs of radius $\eta/2, \eta\in (0,\eta(\zeta))$ around $\lambda$ and around other points in $Z_{m, H}(\zeta)\cap [\Re\lambda -1, \Re\lambda+1]\times [-H, H]$. Call them $B_1,..., B_{M'}$, $M'\le M$, $B_1$ is the one centered at $\lambda$.  We also know that
\begin{equation}
\label{belowphi}
|\phi_{\zeta}(z)|\ge a\, (\eta/2)^M\,,\,\,\forall z\in [\Re\lambda -1, \Re\lambda+1]\times [-H, H] \setminus \cup_{s=1}^{M'} B_s\,.
\end{equation}
Obviously
\begin{align}
\label{dzetaphi}
|\phi_{\zeta}(z) -\phi_{\zeta'}(z)| \le C_0\,|\zeta-\zeta'|\cdot 3^m\,,
\\\,\,\forall \zeta,\zeta'\in \bar D(\theta_0), z\in
[-3^m-1, 3^m+1]\times [-3/2H,3/2H]\,.
\end{align}

Hence, if $\zeta'$ is very close to $\zeta$, namely 
$$
|\zeta'-\zeta|\le  \frac{a}{3C_0}\, (\eta/2)^M\cdot 3^{-m}
$$
we get that 
$$
|\phi_{\zeta}(z)| > |\phi_{\zeta}(z) -\phi_{\zeta'}(z)|\,,\,\,\forall z \in \cup_{s=1}^{M'} \partial B_s\,.
$$
So these functions: $\phi_{\zeta}(z), \phi_{\zeta'}(z)$ have the same number of zeros in each $B_s$ by Rouch\'e's theorem. We need $s=1$, $B_1$ being centered at $\lambda=\lambda_i(\zeta)$.
We conclude that 
$$
\zeta'\in B(\lambda,\eta(\zeta)/2)\rightarrow\lambda(\zeta')
$$
is a continuous function of $\zeta'$ at $\zeta$:
\begin{equation}
\label{changelambda1}
|\lambda(\zeta')-\lambda(\zeta)| \le (3^{m+1}C_0|\zeta'-\zeta|/a)^{1/M}\,.
\end{equation}
Lemma is proved.

\end{proof}

\noindent{\bf Definition.}
$$
\delta_1:= \min\{|\lambda_i(\zeta)-\lambda_j(\zeta)|: i\neq j,\,i,j\le I(\zeta),\, \zeta\in \bar D(\theta_0)\,.
$$

We proved
\begin{equation}
\label{delta1}
\delta_1 > 0\,.
\end{equation}

Given a radial path $p=[\theta_0, \theta_0 + r e^{it}]=[\theta_0,\zeta], r< 2\cdot 3^{-Bm}$ and one $i\le I(\theta_0)$ we can now consider a well-defined and continous function $t\in[0,r]\rightarrow \lambda_i(\theta_0 + t e^{it})$. So we extend $\lambda_i(\theta_0)$ to a  function $\lambda_i(\zeta), \zeta\in D(\theta_0, 2\cdot 3^{-B'm})$, $i=1,...,I(\theta_0)$.
These are all single valued analytic function in this disc. In fact, let us see that a just defined $\lambda_i(\zeta)$ satisfies
\begin{equation}
\label{changelambda2}
|\lambda_i(\zeta)-\lambda_i(\theta_0)| \le C_1\, 3^{-B'm}\,,\,\forall \zeta \in D(\theta_0, 2\cdot 3^{-Bm})\,.
\end{equation}

Suppose \eqref{changelambda2} is already proved. We saw how to extend the analytic germ of $\lambda_i(\cdot)$. By the monodromy theorem we would have a single valued analytic function in $D(\theta_0, 2\cdot 3^{-Bm})$ if we can show that we do not meet branch points while extending to $\zeta\in D(\theta_0, 2\cdot 3^{-Bm})$ along the paths inside $D(\theta_0, 2\cdot 3^{-Bm})$. But if we choose $B'$ large enough, then \eqref{changelambda2} shows that the extension $\lambda_i(\zeta)$ is still in $[-3^m-1,3^m+1]\times [-H,H]$. So by Lemma \ref{compdiscr} we could not meet branch points.

We are left to prove \eqref{changelambda2}. We use Rouch\'e's theorem again. We fix $i\le I(\theta_0)$ and denote $\lambda:=\lambda_i(\theta_0)$ as before in Lemma \ref{cont1}. 

Consider the discs $D_s:=D(\lambda_s(\theta_0), 3^{-B'm})$, $s=1,..., M'$ around zeros of $\phi_{\theta_0}(z)$ lying in
$[\Re\lambda-1, \Re\lambda+1]\times [-H,H]$. Unlike Lemma \ref{cont1} they may be not disjoint. But the number of them is still at most $M$, where $M$ is an absolute constant. Let $\Omega$ be a connected component of $\cup_{s=1}^{M'} D_s$ containing $\lambda$.
Obviously
\begin{equation}
\label{Omega}
\diam\,\Omega \le M\, 3^{-B'm}\,.
\end{equation}
Let $\zeta=\theta_0 + re^{it}$, let $\gamma$ is a continuous path $\lambda_i(\theta_0 + ue^{it}), u\in[0,r], r< 2\cdot 3^{-Bm}$. It starts in $\Omega$, but suppose it hits the boundary of $\Omega$ for $t=t_0$. Denote $\zeta_0=\theta_0+ t_0e^{it}$. We know that

\begin{equation}
\label{below3}
|\phi_{\theta_0}(z)|\ge a\, 3^{-B'Mm}\,,\,\,\forall z\in [\lambda -1, \lambda+1]\times [-H, H] \setminus \cup_{s=1}^{M'} D_s\,.
\end{equation}
Obviously
\begin{align}
\label{dzetaphi3}
|\phi_{\theta_0}(z) -\phi_{\zeta_0}(z)| \le C_0\,|\theta_0-\zeta_0|\cdot 3^m\le 2C_0\,3^m\,3^{-Bm}\,\,,
\\\,\,\forall \zeta_0\in \bar D(\theta_0, 2\cdot 3^{-Bm}), \, z\in
[-3^m-1, 3^m+1]\times [-3/2H,3/2H]\,.
\end{align}

Notice that
$$
\lambda_i(\theta_0 + t_0e^{it})\in [-3^m-1, 3^m+1]\times [-3/2H,3/2H]
$$ 
because of \eqref{Omega}. Then if $B> 10\,B'M$, the fact that $\phi_{\zeta_0}(\lambda_i(\theta_0 + t_0e^{it}))=0$ contradicts the combination of \eqref{below3} and \eqref{dzetaphi3} at $z:=\lambda_i(\theta_0 + t_0e^{it})\in \partial\Omega \subset \cup_{s=1}^{M'} \partial D_s$.

So our continuous  path never hits $\partial \Omega$. Hence for $\zeta=\theta_0 + re^{it}, r< 2\cdot 3^{-Bm}$, the point $\lambda_i(\zeta)\in \Omega$.
Then \eqref{Omega} proves \eqref{changelambda2}. So we have single-valued analytic branches in  $D(\theta_0, 2\cdot 3^{-Bm})$.

\subsection{Estimates of analytic functions $\zeta\in D(\theta_0, 3^{-Bm})\rightarrow \lambda(\zeta)$}
\label{estimates}

We choose the constant $\delta_0>0$  such that 1) in $B(0,\delta_0)$ there are no zeros of any function $\phi_\theta$, 2) $\delta_0< H/10$.

We again fix $\theta_0\in Q\cap\R \setminus 2W_m$, consider the discs $D(\theta_0):=D(\theta_0, 2\cdot 3^{-Bm})$ and $D'(\theta_0):= D(\theta_0, 1.5\,\cdot 3^{-Bm})$. Consider zeros of $\phi_\zeta(z), \zeta\in D(\theta_0), z\in [-3^m, 3^m]\times [-\delta_0/10, \delta_0/10]$, call them $\{\lambda_i(\zeta)\}_{i=1}^{I'(\theta_0)}$, and notice that if $B$ is sufficiently large, then

\begin{equation}
\label{est1}
|\lambda_i(\zeta)-\lambda_i(\theta_0)| \le 3^{-\frac1{20M} Bm}\,, \,\,\forall\zeta\in D(\theta_0),\,i=1,..,I'(\theta_0)\,.
\end{equation}

Here $M$ is an absolute bound on a number of zeros used above. This comes from \eqref{changelambda2} by carefully looking at how we chose $B'$ in \eqref{changelambda2}.

\bigskip

Recall $c_1(\zeta)=\cos ((\zeta)-\frac\pi{2})\,,\,c_2(\zeta)=\cos ((\zeta)-\frac{7\pi}{6})\,,\,c_3(\zeta)=\cos ((\zeta)+\frac{\pi}{6})$.

\noindent{\bf Definition.} Fix $i=1,..., I'(\theta_0)$, put
\begin{eqnarray*}
g_1(\zeta) := g_{1,i}(\zeta) = \frac12 (\lambda_i(\zeta)+\bar\lambda_i(\bar\zeta)) (c_2(\zeta) - c_1(\zeta))\\
g_2(\zeta) := g_{2,i}(\zeta) = \frac12 (\lambda_i(\zeta)+\bar\lambda_i(\bar\zeta)) (c_3(\zeta) - c_1(\zeta))\,.
\end{eqnarray*}

\begin{lemma}
\label{lest2}
$|g_1(\zeta)|, |g_2(\zeta)|, |g_1'(\zeta)|, |g_2'(\zeta)| \le C_2\, 3^{aBm}\,,$
where $a, C_2$ are absolute positive and finite, and $\zeta$ is any point of $D'(\theta_0)$.
\end{lemma}

\begin{proof}
This follows from \eqref{est1}, the fact that $|\lambda_i(\theta_0)|\le 3^m +1$, and the fact that all $g_1, g_2$ are analytic functions in $D(\theta_0)$.

\end{proof}

 Let $D''(\theta_0) :=D(\theta_0, 3^{-Bm})$.

\begin{lemma}
\label{dichot1}
Either
\begin{equation}
\label{nofzeros}
\text{card}\{\zeta\in D''(\theta_0): g_1'(\zeta)=0\} \le B_0(B)\,m\,,
\end{equation}
or 
\begin{equation}
\label{smallderiv}
\|g_1'\|_{L^{\infty}(D''(\theta_0))} \le 3^{-Bm}\,.
\end{equation}
\end{lemma}

\begin{proof}
We have an analytic function $f=g_1'$ in the disc $D'$. It is bounded by $L=C_2 3^{aBm}$. Two things may happen:
at a certain point of $a\in 2/3D'$ it is bigger than $3^{-Bm}$. We write the Jensen's inequality for $\log|f|$ in $D'$. Then the number of zeros in $2/3D'$ will be estimated by $A(\log L-\log(3^{-Bm})$, which is less than $B_0(B)m$ in our case. So lemma's dichotomy is proved.

\end{proof}

 \begin{lemma}
 \label{gg}
 For every $i=1,..., I'(\theta_0)$ and every $\theta\in Q\cap\R$ we have
 $$
 \max (|g_1'(\theta)|, |g_2'(\theta)|) \ge a_1\,\delta_0\,.
 $$
 with a positive absolute $a_1$.
 \end{lemma}

\begin{proof}
We use the notations:
\begin{equation*}
Y(\theta) := Y_{i}(\theta) = \frac12 (\lambda_i(\theta)+\bar\lambda_i(\theta))\,.
\end{equation*}
Also
$g_1'(\theta)=Y'(c_2-c_1)-Y(s_2-s_1)$, and $g_2'(\theta)=Y'(c_3-c_1)-Y(s_3-s_1)$. If $\max (|g_1'(\theta)|, |g_2'(\theta)|)\leq\epsilon$, then it follows that 
$$
|\frac{s_3-s_1}{c_3-c_1}-\frac{s_2-s_1}{c_2-c_1}|\leq\frac{\epsilon (|c_3-c_1|+|c_2-c_1|)}{Y|c_3-c_1||c_2-c_1|}.
$$
 We get $Y\leq C\epsilon$. (One can check that $(c_3-c_1)(c_2-c_1)[\frac{s_3-s_1}{c_3-c_1}-\frac{s_2-s_1}{c_2-c_1}]=3\sqrt{3}/2$.) 
 Recall that $Y(\theta)=\Re\lambda_i(\theta), i=1,..., I'(\theta)$, so $|\Im\lambda_i(\theta)|\le \delta_0/10$. But $|\lambda_i(\theta)|\ge \delta_0$ by the definition of $\delta_0$. So $|Y(\theta)|=|\Re\lambda_i(\theta)| \ge \frac9{10}\delta_0$.
 
\end{proof}

\noindent{\bf Definition.}
$J(\theta_0) := Q\cap\R\cap D(\theta_0, 3^{-Bm})$.

\begin{lemma}
\label{smallderivle0}
If for a given $\theta_0\in Q\cap\R\setminus 2W_m$ in Lemma \ref{dichot1} we have \eqref{smallderiv}, then $|g_2'(\theta) |\ge a_1\delta_0\,,\,\forall \theta\in J(\theta_0)$. 
\end{lemma}

\begin{proof}
Obvious from Lemma \ref{gg}.
\end{proof}

We proved
\begin{lemma}
\label{smallderivle}
For a given $\theta_0\in Q\cap\R\setminus 2W_m$ and $i=1,..., I'(\theta_0)$ we can have either 1) \eqref{nofzeros} for $g_{1,i}'$ and $g_{2,i}'$ simultaneously, or  2) $|g_{1,i}'(\theta) |\ge a_1\delta_0\,,\,\forall \theta\in J(\theta_0)$, or 3) $|g_{2,i}'(\theta) |\ge a_1\delta_0\,,\,\forall \theta\in J(\theta_0)$.
\end{lemma}

Notice that in all three cases
\begin{equation}
\label{mult3}
\max_{x\in\R}\text{card} (J_{\theta_0}\cap g_1^{-1}(x)) \le B_0m\,,\,\,\text{or}\,\,\max_{x\in\R}\text{card} (J_{\theta_0}\cap g_2^{-1}(x)) \le B_0m\,.
\end{equation}
(In the first case both relationships of \eqref{mult3} hold simultaneously.)
In fact, notice that for a smooth real $f$ the cardinality of the pre-image (also called the multiplicity of covering of image) is bounded by the number of zeros of the derivative. In case 2) $g_1'$ does not have any zeros, the same holds in case 3) for $g_2'$. In case 1) we have the bound on the  number of zeros of both $g_1', g_2'$.

\bigskip

In a certain sense \eqref{mult3} is the main claim for the sake of which we launched into investigation of analytic branches of zeros of $\phi_\zeta(z)$. We will need \eqref{mult3} soon, but actually we need a bit more.

\bigskip

Suppose we are {\it not} in the case 2) or 3). Consider functions $p(\zeta)= g_1(\zeta)+g_2(\zeta), m(\zeta)=g_1(\zeta)-g_2(\zeta)$. Then we have the analog of Lemma \ref{dichot1}:

\begin{lemma}
\label{dichot2}
Either
\begin{equation}
\label{mp}
\text{card} \{\zeta\in D'(\theta_0): p'(\zeta)=0\} \le B_0m\,,\,\text{and}\,\text{card} \{\zeta\in D'(\theta_0): m'(\zeta)=0\} \le B_0m\,,
\end{equation}
or
\begin{equation}
\label{mpderiv}
\min(\|m'\|_{L^{\infty}(D'(\theta_0))}, \|p'\|_{L^{\infty}(D'(\theta_0))}) \le 3^{-Bm}\,.
\end{equation}
\end{lemma}

The proof is the same as for Lemma \ref{dichot1}.

If, for example \eqref{mpderiv} happens and, say, $\|m'\|_{L^{\infty}(D'(\theta_0))}\le 3^{-Bm}$, then
$$
||g_1'|-|g_2'|| \le 3^{-Bm}
$$
everywhere on $J(\theta_0)$. Combining this with Lemma \ref{gg}, and choosing large $B$ we conclude that
$$
|g_1'| \ge a_1/2 \delta_0\,\,\text{and}\,\, |g_1'| \ge a_1/2 \delta_0\,\,\text{everywhere on}\,\,J(\theta_0)\,.
$$
So we are back to cases 2) and 3) (simultaneously) of Lemma \ref{smallderivle}.

\bigskip

On the other hand, if \eqref{mp} happens then the number of points $\theta\in J(\theta_0)$ such that
$$
|g_1'(\theta)|=|g_2'(\theta)|
$$ is bounded by $2B_0 m$. In fact, this equality for real $g_1'(\theta), g_2'(\theta)$ may happen only if either
$g_1'(\theta)= g_2'(\theta)$ (so $m'(\theta)=0$) or $g_1'(\theta)=- g_2'(\theta)$ (and so $p'(\theta)=0$). And the number of such points is bounded in \eqref{mp}.

In this latter case, we subdivide $J(\theta_0)$ to intervals $J(s,\theta_0), s=1,..., B_1 m$ according to whether
$$
|g_1'(\theta)|\ge |g_2'(\theta)|\,\,\text{or}\,\, |g_2'(\theta)|\ge |g_1'(\theta)|
$$
everywhere on interval $J_s$.

\bigskip

\begin{theorem}
\label{conclusion}
For every $\theta_0\in Q\cap\R\setminus  2W_m$ we can subdivide interval $J(\theta_0)= \R\cap D(\theta_0, 3^{-Bm})$ into at most $B_1 m$ intervals (may be we even do not need to subdivide at all) $J(s,\theta_0)$ such that on each of them at least one of element of each pair $g_{1,i}(\theta),g_{2,i}(\theta)$, $i=1,...I'(\theta_0)$, is monotone and the modulus of its derivative is at least $a_2\delta_0$.
\end{theorem}

\subsection{The set of small values of $P_{1,\theta}(y)$ revisited}
\label{svr}

Of course
$$
SSV(\theta, \ell) \subset \cup_{s=1}^m SSV(s,\theta, \ell)\,,\,\,SSV(s,\theta, \ell):=\{y\in [0,1]: |\phi_{\theta}(3^s\,y)|< 3^{-A\,\ell}\}\,.
$$

To understand $SSV(s,\theta,\ell)$ let us make some notations first. 

$$
R_k:= [3^k, 3^{k+1}]\times [-\delta_0/10, \delta_0/10],\, k=1,..., m-1\,,\, R_:= [0, 1]\times  [-\delta_0/10, \delta_0/10]\,.
$$
$$
\omega_k(\theta):=\cup_{\lambda(\theta)\in R_k\, \text{is a zero of}\,\phi_\theta} B(\lambda(\theta), 3^{-\ell})\,.
$$
Consider
$$
G_k(\theta) := 3^{-k}\omega_k(\theta) \cup 3^{-k-1}\omega_k(\theta)\cup...\cup 3^{-m}\omega_k(\theta)\,.
$$

Finally, let 
$$
G(\theta):= \cup_{k=0}^m G_k(\theta)\,.
$$

\subsection{Putting the set of small values into a small collection of intervals}
\label{ssv2}

\begin{lemma}
\label{SSVk}
$SSV(s,\theta, \ell)\subset  \cup_{k=1}^s G_k(\theta).$
\end{lemma}

\begin{proof}
Choose $y\in [0,1]$. Suppose $y\in [0,1]\setminus \cup_{k=1}^s G_k$.
Let $k'$ be a number such that $3^s y\in R_{k'}$. Then inevitably $s\ge k'$. If $y\in 3^{-s}\omega_{k'}$, then $y\in G_{k'}$. We asumed the contrary, so $y\notin 3^{-s}\omega_{k'}$.  But then $3^s\,y$ is not in any disc centered at a  zero of $\phi_{\theta}$ in $[y-1, y+1]\times [-\delta_0/10, \delta_0/10]$ and radius $3^{-\ell}$. Other such discs (not counted in $\omega_{k'}(\theta)$) are just far enough from the real axis to contain $3^s\,y\in\R$. Then
$$
|\phi_{\theta}(3^s\,y)| \ge 3^{-A\,\ell}
$$ by Lemma \ref{schke2}.  So $y$ is not in $SSV(s,\theta, \ell)$.

\end{proof}

Thus we trapped the set of small values of $P_{1,\theta}$ into at most $C_3\, 3^m$ intervals:

\begin{lemma}
\label{SSVle}
$SSV(\theta, \ell)\subset  G(\theta).$
\end{lemma}

\bigskip

\noindent{\bf Notations.}
Intervals $J(\theta_0)$ cover the compact $Q\cap\R\setminus 2W_m$. They are all of length $3^{-Bm}$. Choose a finite subcover
$\{J^t:=J(\theta_0^t), t=1,...,T\le 3^{Bm}\}$.
Moreover we will think that $J^t$ are half-open, half-closed intervals making {\it disjoint} cover of $Q\cap\R\setminus 2W_m$.

 For each $t\le T$ we have at most $C_3\cdot 3^m$ distinct analytic functions
$\lambda_{k,t} (\zeta), k=1,...,K(t) \le C_3\cdot 3^m,$ in $D(\theta_0^t, 2\cdot 3^{-Bm})$ which are zeros of $\phi_{\zeta}(z), \zeta\in D(\theta_0^t, 2\cdot 3^{-Bm})$ in $[0,3^m]\times [-\delta_0/10, \delta_0/10]$.

We already considered
\begin{eqnarray*}
Y(k,t,\zeta):= \frac12 (\lambda_{k,t} (\zeta) +\bar\lambda_{k,t} (\bar\zeta))\\
y_s(k,t,\zeta):= 3^{-s}Y(k,t,\zeta),\, s=\kappa, \kappa+1,.., m\\
g_1(k,t,\zeta):= Y(k,t,\zeta)(c_2(\zeta)-c_1(\zeta))\\
g_2(k,t,\zeta):= Y(k,t,\zeta)(c_3(\zeta)-c_1(\zeta))\,.
\end{eqnarray*}
Here $\kappa$ is such that $Y(k,t,\theta)\in (3^{\kappa-1}, 3^{\kappa}]$, and  $\kappa=0$ if $Y(k, t, \theta)\in (0,1]$.

We already proved
\begin{lemma}
\label{Y}
$Y(k,t,\zeta) \ge a_1\delta_0\,, y_s(k,t,\zeta) \ge a_1\delta_0\cdot 3^{-m}\,.$ 
\end{lemma}

We need 

\noindent{\bf Definition.}
$$
j(s,k,t,\theta):=(y_s(k,t, \theta)- 3^{-s}3^{-\ell}, y_s(k,t, \theta)- 3^{-s}3^{-\ell})
\,,s=\kappa,...,m\,.
$$
We proved
\begin{lemma}
\label{SSV3}
If $\theta\in J^t$ then $SSV(\theta,\ell) \subset \cup_{k=1}^{K(t)}\cup_{s=\kappa}^m j(s,k,t,\theta)\,.$
\end{lemma}

\begin{lemma}
\label{R2}
For any $y\in j(s,k,t,\theta)$ we have 
$$
R(y(c_2(\theta)-c_1(\theta)))\le C_4 R(y_s(k,t,\theta)(c_2(\theta)-c_1(\theta)))
$$ and 
$$
R(y(c_3(\theta)-c_1(\theta)))\le C_4 R(y_s(k,t,\theta)(c_3(\theta)-c_1(\theta)))\,.
$$
\end{lemma}

\begin{proof}
The length of the interval $j(s,k,t,\theta)$ is $3^{-s-\ell/2}$, so the diffrence between the last factors in the LHS and the RHS is at most $C_5\,3^{-s-\ell}\cdot 3^{m+\ell/2} $, and because both factors are bounded away from zero by $5/9$ the ration of the last factors in the LHS and RHS diffres from $1$ by at most $C_6\,3^{-s-(\ell/2-m)}$. The second to last factors: the same but their ratio differs from $1$ by at most $C_6\,3^{-s-(\ell/2-m)-1}$.  We continue this comparison, the ratio of the first factors will be different from $1$ by at most $C_6\,3^{-s-(\ell/2-m)-\ell/2}$. Choosing $\ell>2m$ we finish the proof. If we multiply all these ratios we get at most (and at least) an absolute constant.

\end{proof}

\section{Some important standard lemmas}
\label{lemmas}
There are a few important lemmas which we have appealled to repeatedly. The first lemma, Lemma \ref{CET}, uses the Carleson imbedding theorem. Its importance lies in its ability to establish a key relationship between the $L^\infty$ norm of $f_{n,\theta}$ and the $L^2$ norm of $\widehat{f_{n,\theta}}$. This is because the Fourier transform changes the centers of intervals into the frequencies of an exponential polynomial.

The second statement  will be split into Lemmas $\ref{schke1}$, $\ref{schke2}$, and $\ref{schke3}$. They describe standard relationships between a holomorphic function, its zeroes, its boundary values, and its non-zero interior values. Because we use them so often, we have taken the trouble of stating and proving them so as to streamline the main argument of the paper.

\subsection{{\bf A corollary of the Carleson imbedding theorem}}
\begin{lemma}
\label{CET}
Let $j=1,2,...k$, $c_j\in\C$, $|c_j|=1$, and $\alpha_j\in\R$. Let $A:=\lbrace \alpha_j\rbrace_{j=1}^k$. Then
$$
\int_0^1{|\sum_{j=1}^k{c_je^{i\alpha_jy}}|^2dy}\leq C\,k\cdot \sup_{I\text{ a unit interval}}\card \{A\bigcap I\}\,.
$$
\end{lemma}

\begin{proof}
Let $A_1:=\{\mu= \alpha +i: \alpha\in A\}$. Let $\nu:=\sum_{\mu\in A_1} \delta_{\mu}$. This is a measure in $\C_+$. Obviously its Carleson constant
$$
\|\nu\|_C :=\sup_{J\subset \R,\, J\,\text{is an interval}}\frac{\nu(J\times [0,|J|])}{|J|}
$$ can be estimated as follows
\begin{equation}
\label{CETeq}
\|\nu\|_C \le 2\,\sup_{I\text{ a unit interval}}\card \{A\bigcap I\}\,.
\end{equation}

Recall that
\begin{equation}
\label{CETeq1}
\forall f\in H^2(\C_+)\,\,\int_{C_+} |f(z)|^2 \,d\nu(z) \le C_0\, \|\nu\|_C\|f\|_{H^2}^2\,,
\end{equation}
where $C_0$ is an absolute constant.
Now we compute
$$
\int_0^1{|\sum_{j=1}^k{c_je^{i\alpha_jy}}|^2dy}\leq e^2\int_0^1{|\sum_{j=1}^k{c_je^{i(\alpha_j+i)y}}|^2dy}\leq
$$
$$
 e^2\int_0^{\infty}{|\sum_{j=1}^k{c_je^{i(\alpha_j+i)y}}|^2dy} = e^2\int_{\R}|\sum_{\mu\in A_1}\frac{c_{\mu}}{x-\mu}|^2\,,
 $$
 where $c_{\mu} := c_j$ for $\mu= \alpha_j +i$. The last equality is by Plancherel's theorem.
 
 We continue
$$
\int_{\R}|\sum_{\mu\in A_1}\frac{c_{\mu}}{x-\mu}|^2 =\sup_{f\in H^2(C_+),\, \|f\|_2\le 1}\bigg|\langle f,\sum_{\mu\in A_1}\frac{c_{\mu}}{x-\mu}\rangle\bigg|^2=
$$
$$
 4\pi^2\sup_{f\in H^2(C_+),\, \|f\|_2\le 1}|\sum_{\mu\in A_1}c_{\mu}f (\mu)|^2\le C\,\card\{A_1\}\sup_{f\in H^2(C_+),\, \|f\|_2\le 1}\sum_{\mu\in A_1} |f(\mu)|^2 \le
 $$
 $$
 C\,\card\{A\}\sup_{f\in H^2(C_+),\, \|f\|_2\le 1}\int_{C_+}|f(z)|^2\,d\nu(z)\le 
 2C_0C\,\card\{A\}\,\sup_{I\text{ a unit interval}}\card \{A\bigcap I\}\,.
 $$
 This is by \eqref{CETeq1} and \eqref{CETeq}. The lemma is proved.

\end{proof}

\subsection{{\bf A Blaschke estimate}}
\begin{lemma}\label{schke1}
Let $D$ be the closed unit disc in $\C$. Suppose $\phi$ is holomorphic in an open neighborhood of $D$, $|\phi(0)|\geq 1$, and the zeroes of $\phi$ in $\frac{1}{2}D$ are given by $\lambda_1,\lambda_2,...,z_M$. Let $C=||\phi||_{L^\infty(D)}$. Then $M\leq \log_2(C).$
\end{lemma}
\begin{proof}
Let 
$$
B(z)=\prod_{k=1}^M{\frac{z-\lambda_k}{1-\bar{\lambda_k}z}}.
$$
 Then $|B|\leq 1$ on $D$, with $=$ on the boundary. If we let $g:=\frac{\phi}{B}$, then $g$ is holomorphic and nonzero on $\frac{1}{2}D,$ and $|g(e^{i\theta})|\leq C$ $\forall\theta\in [0,2\pi]$. Thus $|g(0)|\leq C$ by the maximum modulus principle. So we have $$C\geq |g(0)|=\frac{|\phi(0)|}{|B(0)|}\geq\prod_{k=1}^M{\frac{1}{|\lambda_k|}}\geq 2^M.$$
\end{proof}

\begin{lemma}\label{schke2}
In the same setting as Theorem $\ref{schke1}$, the following is also true for all $\delta\in (0,1/3)$: $\lbrace z\in\frac{1}{4}D:|\phi|<\delta\rbrace\subseteq\bigcup_{1\leq k\leq M} B(\lambda_k,\e)$, where $$\e:=\frac{9}{16}(3\delta)^{1/M}\leq\frac{9}{16}(3\delta)^{1/\log_2(C)}.$$
\end{lemma}
\begin{proof}
Let $\delta\in (0,1/3)$, and let $z\in\frac{1}{4}D$ such that $|z-\lambda_k|>\e\,\,\forall k$. Note that $g$ is harmonic and nonzero on $\frac{1}{2}D$ with $|g(0)|\geq 2^M$. Thus Harnack's inequality ensures that $|g|\geq\frac{1}{3}2^M$ on $\frac{1}{4}D$, so there 
$$
|\phi(z)|\geq |g(z)B(z)| \geq\frac{1}{3}2^M\prod_{k=1}^M{|\frac{z-\lambda_k}{1-\bar{\lambda_k}z}|}\geq(\frac{16\e}{9})^M\frac{1}{3}=\delta.
$$
 We can conclude the proof by the contrapositive.
\end{proof}

\begin{lemma}\label{schke3}
Let $\delta\in (0,1/3)$. Let $\phi$ be holomorphic in the horizontal strip $\R\times(-14H,14H)$, with $H$ large enough, and with $|\phi|\leq C$ in the strip. Let $Box=[x-1,x+1]\times[-H,H].$ Let $max_{z\in Box}|\phi|\geq 1$. Let $\e$ be as in Theorem $\ref{schke2}$, and call the zeroes of $\phi$ in an $\e$ neighborhood of $Box$ by the names $\lambda_1,\lambda_2,...,\lambda_M$. Then $M\leq \log_2(C)$, and 
$$
\lbrace z\in Box:|\phi (z)|\leq\delta\rbrace\subseteq\bigcup_{1\leq k\leq M} B(\lambda_k,\e).
$$
\end{lemma}
\begin{proof}
Take $D$ to be a disc of radius $12\,H$ centered at the $z$ which maximizes $|\phi|$ in $Box$.
\end{proof}

\section{Discussion}
\label{discu}

 By replacing our $L^{\infty}$ estimates of $f_{n,\theta}(x)$ by their $L^2$ estimates as in \cite{NPV} we possibly could have improved the estimate \eqref{gooddir} to
 $\mathcal{L}_{\theta, N^C} \le C/K$ by being more restrictive in choosing good directions (now the good direction would mean that there are many stacks of $K$ elements overlapping when projected to $\theta$-direction, or alternatively speaking that the set of overlapping has a considerable measure). This would  improve Theorem \ref{mainth} to  
 
$$
\int_0^{\pi}|\mathcal{L}_{\theta, n}|\,d\theta \le C\, e^{-c\,\sqrt{\log\,n}}\le \frac{C}{\log\,n}\,.
$$

This is still one order far from the power law of \cite{NPV}. Of course the power law is correct. But to improve the last estimate to a power law one would need to sort out the zeros of our trigonometric polynomials in a more careful way than we did or to come up with a different idea.

  \bibliographystyle{amsplain}

\begin{thebibliography}{99}
 
  \bibitem{BV} M. Bateman, A.Volberg, {\em An estimate from below for the Buffon needle probability of the four-corner Cantor set}, arXiv:math. 0807.2953v1, 2008, pp. 1-11. To appear in Math. Research Letters.
  \bibitem{BK} M. Bateman, N.Katz, {\em Kakeya sets in Cantor directions},  arXiv:math. 0609187v1, 2006, pp. 1--10.
  \bibitem{B} M. Bateman, {\em Kakeya sets and the directional maximal operators in the plane},  arXiv:math.CA 0703559v1, 2007, pp. 1--20.
  \bibitem{besi} A. S. Besicovitch, {\em Tangential properties of sets and arcs of infinite linear measure}, {\em Bull.\ Amer.\ Math.\ Soc.}  {\bf 66} (1960), 353--359.
 \bibitem{BV2}
	{\sc M. Bond, A. Volberg}: {\sl Estimates from below of the Buffon noodle probability for undercooked noodles}, arXiv:math/0811.1302v1, 2008, pp. 1--10.
  \bibitem{falc1} K. J. Falconer,  The geometry of fractal sets.
  Cambridge Tracts in Mathematics, 85. C.U.P., Cambridge--New York, (1986).
  \bibitem{keich} U. Keich, On $L^p$ bounds for Kakeya maximal functions and the  Minkowski dimension in $\R^2$, {\em Bull. London. Math. Soc.} {\bf 31} (1999),  pp. 213--221.
  \bibitem{kenyon} R. Kenyon,  {\em Projecting the one-dimensional Sierpinski gasket, } {\em Israel J. Math.} {\bf 97} (1997), 221--238.
  \bibitem{LZ} I. Laba, K. Zhai, {\em Favard length of product Cantor sets}, arXiv:0902:0964v1, Feb. 5 2009.
  \bibitem{lawang} J. C. Lagarias and Y. Wang, {\em Tiling the line with translates of one tile}, {\em Invent. Math.}{\bf 124} (1996), 341--365.
  \bibitem{mattila1} P. Mattila, {\em Orthogonal projections, Riesz capacities and Minkowski content}, {\em Indiana Univ.\ Math.\ J.}\ {\bf 39} (1990), 185--198.
   \bibitem{mattila2} P. Mattila, {\em Hausdorff dimension, projections, and the Fourier tarnsform}, Publ. Mat.,  {\bf 48} (2004), pp. 3--48.
  \bibitem{mattila3} P. Mattila, {\em Geometry of Sets and Measures in Euclidean Spaces}, Cambridge University Press, 1995.
 \bibitem{NPV} F. Nazarov, Y. Peres, A. Volberg {\em  The power law for the Buffon needle probability of the four-corner Cantor set}, arXiv:0801.2942, 2008, pp. 1--15.
  \bibitem{pesiso} Y. Peres, K. Simon and B. Solomyak, {\em Self-similar sets of zero Hausdorff measure and positive packing measure}, {\em Israel J.\ Math.} {\bf 117} (2000),353--379.
  \bibitem{PS} Y. Peres and B. Solomyak, {\em How likely is {B}uffon's needle to fall near a planar {C}antor set?} \newblock {\em Pacific J. Math. 204}, 2 (2002), 473--496.
  \bibitem{schoenberg} I. J. Schoenberg, {\em On the Besicovitch-Perron solution of the Kakeya problem}, {\em Studies in mathmatical analysis and related topics}, 
  \bibitem{T} T. Tao, {\em A quantitative version of the Besicovitch projection theorem via multiscale analysis,} pp. 1--28, arXiv:0706.2446v1 [math.CA] 18 Jun 2007.
  \end{thebibliography}

\end{document}